\documentclass[11pt,letterpaper]{amsproc}
\usepackage{latexsym,amscd,amssymb,url}
\usepackage{extarrows}
\usepackage{verbatim}
\usepackage{dsfont}

\usepackage{setspace}
\usepackage[bottom=3.5cm, right=4.5cm, left=4cm, top=2.cm]{geometry} 

\usepackage{lipsum}

\newcommand\blfootnote[1]{%
  \begingroup
  \renewcommand\thefootnote{}\footnote{#1}%
  \addtocounter{footnote}{-1}%
  \endgroup
}

\usepackage[all,cmtip]{xy}  
\usepackage{graphicx}
\usepackage{xcolor}
\usepackage{newpxmath}
\usepackage{enumitem} 
\definecolor{dblue}{rgb}{0,0,.6}
\usepackage[colorlinks=true, linkcolor=dblue, citecolor=dblue, filecolor = dblue, menucolor = dblue, urlcolor = dblue]{hyperref}

\usepackage{mathtools}

\newcommand{\set}[1]{\left\{ #1 \right\}}

\newcommand{\can}{\textnormal{can}}

\newcommand{\Sch}{\mathsf{Sch}}

\newcommand{\ZZ}{\mathbb{Z}}

\newcommand{\id}{\textnormal{id}}

\newcommand{\Hom}{\textnormal{Hom}}

\newcommand{\Gal}{\textnormal{Gal}}

\newcommand{\Spec}{\textnormal{Spec}}

\newcommand{\ca}[1]{{\mathcal{#1}}}

\newcommand{\mr}[1]{{\mathscr{#1}}}

\newcommand{\tn}[1]{{\textnormal{#1}}}
\let\rm\relax 
\newcommand{\rm}[1]{{\mathrm{#1}}}

\DeclareSymbolFont{bbm}{U}{bbm}{m}{n}
\DeclareSymbolFontAlphabet{\mathbbm}{bbm}

\let\rm\relax 
\newcommand{\rm}[1]{{\mathrm{#1}}}

\numberwithin{equation}{section}

\newtheorem{theorem}{Theorem}[section]

\theoremstyle{plain}

\newtheorem{corollary}[theorem]{Corollary}

\newtheorem{lemma}[theorem]{Lemma}

\newtheorem{proposition}[theorem]{Proposition}

\theoremstyle{definition}

\newtheorem{definition}[theorem]{Definition}
\newtheorem{example}[theorem]{Example}

\newcommand*{\triple}[2][.1ex]{%
  \mathrel{\vcenter{\offinterlineskip%
  \hbox{$#2$}\vskip#1\hbox{$#2$}\vskip#1\hbox{$#2$}}}}

  \newcommand*{\triplerightarrow}{\triple{\rightarrow}}

\newcommand{\refsymbol}{{\ensuremath \Swarrow }}

\usepackage[backend=bibtex,style=alphabetic]{biblatex}

\begin{document}    

\title[Descent for algebraic stacks]{Descent for algebraic stacks}

\author{Olivier de Gaay Fortman} 
\address{Department of Mathematics, Utrecht University, Budapestlaan 6, 3584 CD Utrecht, The Netherlands}
\email{a.o.d.degaayfortman@uu.nl}

\maketitle

\date{\today}

\begin{abstract} 
We prove that algebraic stacks satisfy $2$-descent for fppf coverings. We generalize Galois descent for schemes to stacks, by considering the case where the fppf covering is a finite Galois covering, and reformulating $2$-descent data in terms of Galois group actions on the stack. 
\end{abstract}

\section{Introduction} \label{section:introduction}
\blfootnote{\today} Let $S' \to S$ be a morphism of affine schemes, faithfully flat and locally of finite presentation. By a theorem of Grothendieck, the functor $X \mapsto X \times_SS'$ 
induces an equivalence of categories between the category of $S$-schemes $X$ and the category of pairs $(X', \phi)$ where $X'$ is an $S'$-scheme and $\phi$ a descent datum for $X'$ over $S'$ such that $X'$ admits an open covering by $S'$-affine schemes which are stable under $\phi$ (cf.\ \cite{groth-descente}). In the case when $S = \Spec(k)$, $S' = \Spec(k')$, and the morphism $S' \to S$ corresponds to a finite Galois extension of fields $k \subset k'$, this is known as Galois descent, %
and was proven by Weil (cf.\ \cite{weilgalois}). 


The aim of this paper is to present the most natural analogue of this result in the setting of algebraic stacks. To do this, we build on existing work in descent theory for stacks, especially results by Giraud~\cite{giraud}, Duskin~\cite{Duskin}, and Breen~\cite{breen}. Our main contributions are as follows:
\begin{enumerate}
    \item We provide a reference for descent theory for stacks, explaining how the aforementioned results imply that stacks with descent data over a scheme \( S' \) descend along fppf morphisms \( p \colon S' \to S \).
    \item We show that algebraic stacks descend not just as stacks, but as algebraic stacks; we prove the same for Deligne--Mumford stacks. 
    \item We generalize Galois descent from schemes to stacks by considering the case where the fppf cover is a finite Galois cover, and by attaching 2-descent data to actions of the Galois group on the stack.
\end{enumerate}

\subsection{Main results} In the case of stacks, the analogue of the aforementioned descent-theory for schemes is a notion called \emph{$2$-descent}, which appears to have been introduced by Duskin \cite{Duskin}. 
As it turns out, $2$-descent data for algebraic stacks with respect to faithfully flat locally finitely presented morphisms of schemes are always effective.
More precisely, we have the following result. 
For a scheme $S$, let $(\rm{Sch}/S)_{fppf}$ be the big fppf site of $S$ as in \cite[\href{https://stacks.math.columbia.edu/tag/021S}{Tag 021S}]{stacks-project}; a \emph{stack over $S$} is a stack in groupoids $\ca X \to  (\rm{Sch}/S)_{fppf}$, see \cite[\href{https://stacks.math.columbia.edu/tag/0304}{Tag 0304}]{stacks-project}. 

\begin{theorem}  \label{theorem:main:intro}
Let $S' \to S$ be a faithfully flat locally finitely presented morphism of schemes, and let $\ca X'$ be a stack over $S'$. Let $(\phi, \psi)$ be a $2$-descent datum for the stack $\ca X'$ over $S'$, see Definition \ref{definition:2-descent}. Then the following holds. 
\begin{enumerate}
\item
The $2$-descent datum $(\phi, \psi)$  is effective. That is, there exists a stack $\ca X$ over $S$, an isomorphism of stacks over $S'$,
\[
\rho \colon \ca X \times_S S' \xrightarrow{\sim} \ca X',
\]
and a $2$-isomorphism $\chi \colon p_2^\ast \rho \circ \textnormal{can} \Rightarrow \phi \circ p_1^\ast \rho$ as in the following diagram:
\begin{align*} 
\xymatrix{
p_1^\ast(\ca X \times_S S') \ar[r]^{\textnormal{can}} \ar@{}[dr] | {\refsymbol} \ar[d]^{p_1^\ast \rho} & p_2^\ast(\ca X \times_S S') \ar[d]^{p_2^\ast \rho}\\
p_1^\ast \ca X' \ar[r]^\phi & p_2^\ast \ca X',
}
\end{align*}
such that the natural compatibility between $\chi$ and $\psi$ is satisfied.
\item 
The stack $\ca X'$ over $S'$ is an algebraic stack over $S'$ if and only if the stack $\ca X$ over $S$ is an algebraic stack over $S$.
\item 
The stack $\ca X'$ over $S'$ is a Deligne--Mumford stack over $S'$ if and only if the stack $\ca X$ over $S$ is a Deligne--Mumford stack over $S$.
\end{enumerate}
\end{theorem}


Note that even the case where $\ca X'$ is a scheme seems to yield a non-trivial result (cf.\ Corollary \ref{cor:corollary}). Of course, in some sense these results are not surprising: the descended stack $\ca X$ is obtained by defining $\ca X(T)$ as the groupoid of objects of $\ca X'(T \times_SS')$ equipped with a descent datum relative to the $2$-descent datum of $\ca X'$, for any scheme $T$ over $S$. More precisely, the first assertion in the above theorem follows from:
\begin{theorem}[Breen, Giraud] \label{thm:breen}
 Consider the $2$-fibred category $$\underline{Stack}_S \to (\rm{Sch}/S)_{fppf},$$ whose fibre over $U \in (\rm{Sch}/S)_{fppf}$ is the $2$-category $\underline{Stack}(U)$ of stacks over $U$. Then $\underline{Stack}_S$ is a $2$-stack over $S$.
\end{theorem}
\begin{proof}
See \cite[Example 1.11.(1)]{breen} and \cite[Chapitre II, \S 2.1.5]{giraud}. 
\end{proof}

The other two assertions in Theorem \ref{theorem:main:intro} follow from the fact that the property of a stack of being algebraic (resp.\ Deligne--Mumford) is local on the base for the fppf (resp.\ \'etale) topology; see Lemma \ref{lemma:local} for a precise statement. 
For details, we refer to Section \ref{sec:descending-stacks}. 

Assume that $S' \to S$ is a finite surjective étale morphism of schemes which is a Galois covering with Galois group $\Gamma$, acting on the left on $S'$. 
Then for a stack $\ca X'$ over $S'$, one can reformulate the notion of $2$-descent datum for $\ca X'$ over $S'$ in terms of an action of $\Gamma$ on $\ca X'$ over the action of $\Gamma$ on $S'$ over $S$, as in the classical case. To explain this, 
for an element $\sigma \in \Gamma$, define ${^\sigma \ca X'}$ as the fibre product $\ca X' \times_{S', \sigma} S'$, viewed as a stack over $S'$ whose structure morphism ${^\sigma \ca X'} \to S'$ is the pull-back of the morphism $\ca X' \to S'$ along $\sigma \colon S' \to S'$.
\begin{definition}\label{def:Galois2descent}Let $S' \to S$ be a finite surjective étale morphism of schemes which is a Galois covering with Galois group $\Gamma$. Let $\ca X'$ be a stack over $S'$. 
A \emph{Galois $2$-descent datum} consists of: 
\begin{enumerate}
\item a family of $1$-isomorphisms of $S'$-stacks $f_\sigma \colon {^\sigma \ca X'} \xrightarrow{\sim} \ca X'$ $(\sigma \in \Gamma)$; 
\item a family of $2$-isomorphisms $\psi_{\tau, \sigma} \colon f_\sigma \circ {^\sigma f_\tau} \Rightarrow f_{\tau\sigma}$ $((\tau, \sigma) \in \Gamma \times \Gamma)$; 
\end{enumerate}
such that for each $(\gamma, \tau, \sigma) \in \Gamma \times \Gamma \times \Gamma$, the diagram of $2$-morphisms
\begin{align} \label{galoisdiagram2morphisms}
\begin{split}
\xymatrixcolsep{5pc}
\xymatrixrowsep{2.8pc}
\xymatrix{
f_\sigma \circ {^\sigma f_\tau} \circ {^{\tau \sigma}f_\gamma} \ar@{=>}[d]^{{f_\sigma}_\ast(^\sigma \psi_{\gamma, \tau})} \ar@{=>}[r]^-{(^{\tau \sigma}f_\gamma)^\ast(\psi_{\tau, \sigma})}& f_{\tau\sigma} \circ {^{\sigma \tau }f_\gamma}\ar@{=>}[d]^-{\psi_{\gamma, \tau \sigma}} \\
f_\sigma \circ {^\sigma f_{\gamma \tau}}\ar@{=>}[r]^-{\psi_{\gamma\tau, \sigma}}& f_{\gamma \tau \sigma}
}
\end{split}
\end{align}
is commutative. \end{definition}
A group action (in the sense of \cite[Definition 1.3]{romagny}) of the finite group $\Gamma$ on the stack $\ca X'$ over $S$, which is compatible with the action of $\Gamma$ on $S'$ over $S$, gives rise to a Galois $2$-descent datum for $\ca X'$, see Lemma \ref{lemma:action-and-descent}. Moreover, giving a Galois $2$-descent datum for $\ca X'$ is equivalent to giving $2$-descent datum for $\ca X'$, see Lemma \ref{lemma:set-can-bij}. As a corollary of Theorem \ref{theorem:main:intro}, one therefore obtains:
\begin{theorem} \label{theorem:galois}
Let $S' \to S$ be a finite surjective étale morphism of schemes which is a Galois covering with Galois group $\Gamma$. Let $\ca X'$ be an algebraic stack over $S'$, equipped with a Galois $2$-descent datum $( f_\sigma$ $(\sigma \in \Gamma)$, $\psi_{\tau, \sigma}$ $(\tau, \sigma \in \Gamma))$. Then there exists an algebraic stack $\ca X$ over $S$ and an isomorphism $\rho \colon \ca X \times_S S' \xrightarrow{\sim} \ca X'$ of stacks over $S'$. The stack $\ca X$ is Deligne--Mumford if and only if $\ca X'$ is. 
\end{theorem}

Observe that the statement in Theorem \ref{theorem:galois} can be made a bit more precise. Namely, with notation and assumptions as in the theorem, there exists an isomorphism of stacks $\rho \colon \ca X \times_S S' \xrightarrow{\sim} \ca X'$ over $S'$ as well as a family of $2$-isomorphisms $\chi_\sigma \colon \rho \circ \rm{can} \Rightarrow f_\sigma \circ {^\sigma \rho}$ for $\sigma \in \Gamma$ as in the following diagram:
\[
\xymatrix{
{^\sigma \left( \ca X \times_S S' \right)} \ar[r]^-{\rm{can}} \ar[d]_-{^\sigma \rho} & \ca X \times_S S' \ar[d]^-\rho  \ar@{}[dl] | {\refsymbol \; \chi_\sigma} \\
{^\sigma \ca X'} \ar[r]_-{f_\sigma} & \ca X',
}
\]
such that the obvious compatibility conditions are satisfied. 

Returning to the case of an arbitrary faithfully flat locally finitely presented morphism of schemes $S' \to S$, Theorem \ref{thm:breen} shows that the canonical $2$-functor $$\underline{Stack}(S) \to \underline{Stack}(\set{S' \to S}),$$ 
that sends a stack over $S$ to the associated stack $\ca X'$ with canonical $2$-descent datum over $S'$, is an equivalence of $2$-categories.  Here, $\underline{Stack}(\set{S' \to S})$ is the $2$-category of stacks over $S'$ equipped with a $2$-descent datum (see Definitions \ref{definition:2-descent} and \ref{def:morphism}). With regard to morphisms, this has the following consequence.
\begin{proposition} \label{proposition:morphisms}
Let $S' \to S$ be a faithfully flat locally finitely presented morphism of schemes. For $i = 1,2$, let $\ca X_i$ be a stack over $S$, and let $\ca X_i'$ be the base change of $\ca X_i$ along $S' \to S$ equipped with its canonical $2$-descent datum $(\phi_i, \psi_i)$. 
Then the canonical functor
$$
\Hom_S(\ca X_1, \ca X_2) \to \Hom_{\rm{descent}/S'}\left(\left( \ca X_1', \phi_1, \psi_1 \right), \left(\ca X_2', \phi_2, \psi_2 \right) \right)
$$
is an equivalence of categories. 
\end{proposition}

Here, $\Hom_S(\ca X_1, \ca X_2)$ denotes the category of morphisms of stacks $\ca X_1 \to \ca X_2$ over $S$, and $\Hom_{\rm{descent}/S'}(( \ca X_1', \phi_1, \psi_1 ), (\ca X_2', \phi_2, \psi_2 ) )$ the category of morphisms of stacks with $2$-descent data over $S'$ (see Definition \ref{def:morphism}). 

\section*{Acknowledgements}
The author thanks Emiliano Ambrosi and Marta Pieropan for useful discussions and comments, and he thanks the referee for his or her careful reading and valuable comments. This project has received funding from the European Research Council (ERC) under the European Union’s Horizon 2020 research and innovation programme under grant agreement N\textsuperscript{\underline{o}} 948066 (ERC-StG RationAlgic), and 
from 
the ERC Consolidator Grant FourSurf N\textsuperscript{\underline{o}} 101087365.

\section{Descending schemes} \label{section:descendingschemes}
Let 
$
p \colon S' \to S
$
be a morphism of schemes which is faithfully flat and locally of finite presentation. We get a diagram
\[
S^{''} \coloneqq S' \times_{S} S' \overset{p_1}{\underset{p_2}{\rightrightarrows}} S' \to S,
\]
and if $S^{'''} = S' \times_S S' \times_S S'$, we can extend this to the diagram
\[
S^{'''} \triplerightarrow S^{''} \rightrightarrows S^{'} \to S
\]
where the three arrows $S^{'''} \to S^{''}$ are $p_{12}$, $p_{13}$ and $p_{23}$. 

Let $X'$ be a scheme over $S'$. Define 
\[
p_i^\ast X' = X' \times_{S', p_i}S^{''}, \quad p_{jk}^\ast p_i^\ast X' = \left(p_i^\ast X'\right) \times_{S^{''}, p_{jk}} S^{'''}
\]
and note that
$
p_{jk}^\ast p_i^\ast X' = \left(p_i^\ast X'\right) \times_{S^{''}, p_{jk}} S^{'''} = \left(p_i \circ p_{jk}\right)^\ast X'. 
$

Recall that a \emph{descent datum} for $X'/S'$ is an $S^{''}$-isomorphism 
\[
\phi \colon p_1^\ast X' \xrightarrow{\sim} p_2^\ast X'
\]
such that the following diagram commutes:
\[
\xymatrixcolsep{3pc}
\xymatrixrowsep{2.7pc}
\xymatrix{
p_{12}^\ast p_1^\ast X' \ar[r]^{p_{12}^\ast \phi}\ar@{=}[d] & p_{12}^\ast p_2^\ast X' \ar@{=}[r]& p_{23}^\ast p_1^\ast X' \ar[d]^{p_{23}^\ast\phi} \\
p_{13}^\ast p_1^\ast X' \ar[r]^{p_{13}^\ast \phi}&p_{13}^\ast p_2^\ast X&p_{23}^\ast p_2^\ast X'. \ar@{=}[l] 
}
\]
In other words, one requires that 
\[
p_{23}^\ast\phi \circ p_{12}^\ast \phi = p_{13}^\ast \phi \quad \text{ as morphisms } \quad p_{12}^\ast p_1^\ast X' \to p_{13}^\ast p_2^\ast X'.
\]

\begin{theorem}[Grothendieck] \label{theorem:grothendieckdescent}
Let $p \colon S' \to S$ be a faithfully flat locally finitely presented morphism of schemes. The functor
    $
    X \mapsto p^\ast X
    $
defines an equivalence of categories between the category of $S$-schemes $X$ and the category of pairs $(X', \phi)$ where $X'$ is an $S'$-scheme and $\phi$ a descent datum for $X'/S'$ such that $X'$ admits an open covering by $S'$-affine schemes stable under $\phi$.     
\end{theorem}
\begin{proof}
See \cite[Theorem 2]{groth-descente} and the discussion below Lemma 1.2 in \emph{loc.\ cit.}
\end{proof}


Next, recall how to make this explicit in the case when $S' \to S$ is a finite surjective étale morphism of schemes which is a Galois covering with Galois group $\Gamma$. For instance, $S$ could be the spectrum of a field $k$, $S'$ the spectrum of a finite Galois extension $k' \supset k$, and $\Gamma$ the Galois group of $k'/k$. Let $X'$ be a scheme over $S'$ and call a \emph{Galois descent datum} any set of isomorphisms 
\[
f_\sigma \colon ^\sigma X' \xrightarrow{\sim} X'
\]
of schemes over $S'$, indexed by $\sigma \in \Gamma$, satisfying the condition that
\[
f_{\tau \sigma} = f_\sigma \circ {^\sigma(f_\tau)} \quad \text{as isomorphisms} \quad ^{\tau\sigma}X' \xrightarrow{\sim} {^\sigma X'} \xrightarrow{\sim}X', \quad \forall \tau, \sigma \in \Gamma.
\]
An action of $\Gamma$ on $X'$ as a scheme over $S$ is said to be \emph{compatible with the action of $\Gamma$ on $S'$ over $S$} if for each $\sigma \in \Gamma$, the composition $X' \xrightarrow{\sigma} X' \to S'$ agrees with the composition $X' \to S' \xrightarrow{\sigma} S'$. 


\begin{lemma} 
Let $p \colon S' \to S$ be a finite surjective étale morphism of schemes. Assume $\pi$ is a Galois covering with Galois group $\Gamma$. Let $X'$ be a scheme over $S'$. 
\begin{enumerate}
\item
Giving a descent datum for $X'$ over $S'$ is equivalent to giving a Galois descent datum for $X'$ over $S'$. 
\item These notions are further equivalent to giving  an action of $\Gamma$ on $X'$ over $S$ compatible with the action of $\Gamma$ on $S'$ over $S$. 
\end{enumerate}
\end{lemma}
\begin{proof}
This is well-known; see e.g.\ \cite[Section 6.2, Example B]{BLR} and \cite[Proposition 4.4.4]{poonen}. 
\end{proof}


\section{Descending algebraic stacks} \label{sec:descending-stacks}

\subsection{Descent for algebraic stacks}
Let $p \colon S' \to S$ be a faithfully flat locally finitely presented morphism of schemes. Let $\mathcal X'$ be a stack in groupoids over $S'$, in the sense of \cite[\href{https://stacks.math.columbia.edu/tag/0304}{Tag 0304}]{stacks-project}. 
Let 
\[
S^{''''} = S' \times_S S' \times_S S' \times_S S';
\]
it is equipped with four projections \begin{align} \label{equation:ri} r_i \colon S^{''''} \to S'.\end{align} Similarly, $S^{'''}$ is equipped with three projections $q_i \colon S^{'''} \to S'$. Note that there are canonical isomorphisms
$
p_{12}^\ast p_1^\ast \ca X'  = (p_1 \circ p_{12})^\ast \ca X' = q_1^\ast \ca X'.
$
Similarly, there are canonical isomorphisms
$
p_{123}^\ast p_{12}^\ast p_1^\ast = \left( p_1 \circ p_{12} \circ p_{123}\right)^\ast = 
r_1^\ast \ca X',
$
of algebraic stacks on $S'$. One has similar isomorphisms relating the other $p_{ijk}^\ast p_{\alpha\beta}^\ast p_\nu^\ast \ca X'$ with $r_\mu^\ast \ca X'$, for $i,j,k \in \set{1,2,3,4}$, $\alpha, \beta \in \set{1,2,3}$, $\nu \in \set{1,2}$ and $\mu \in \set{1,2,3,4}$. 

Consider an isomorphism of $S^{''}$-stacks (i.e.~an equivalence of $\tn{Sch}/S^{''}$-categories):
\[
\phi \colon p_1^\ast \ca X' \to p_2^\ast \ca X',
\]
and let $\psi$ be a $2$-morphism 
$$
\psi \colon p_{23}^\ast \phi \circ p_{12}^\ast \phi \Rightarrow p_{13}^\ast\phi,
$$
which we may picture as the $2$-morphism $\Rightarrow$ in the following diagram:
\begin{small}
\begin{align} \label{diagram:stacks-1}
\begin{split}
\xymatrixcolsep{3pc}	
\xymatrixrowsep{3pc}
\xymatrix{
q_1^\ast \ca X' \ar@{=}[r]  &p_{12}^\ast p_1^\ast \ca X' \ar[r]^{p_{12}^\ast \phi}\ar@{=}[d] & p_{12}^\ast p_2^\ast \ca X' 
\ar@{=}[r] \ar@{}[dr] | {\refsymbol \; \psi\;\;\;\;\;\;\;\;\;\;\;\;\;\;\;\;\;\;} & p_{23}^\ast p_1^\ast \ca X' \ar[d]^{p_{23}^\ast\phi} \ar@{=}[r] & q_2^\ast \ca X'  \\
&p_{13}^\ast p_1^\ast \ca X' \ar[r]^{p_{13}^\ast \phi}&p_{13}^\ast p_2^\ast \ca X'&p_{23}^\ast p_2^\ast \ca X' \ar@{=}[l] \ar@{=}[r] & q_3^\ast \ca X'.
}
\end{split}
\end{align}
\end{small}
Consider the four maps 
\[
p_{123}, p_{124}, p_{134}, p_{234} \colon S^{''''} \to S^{'''},
\]
and note that
\begin{align*}
p_{123}^\ast \left( p_{23}^\ast \phi \circ  p_{12}^\ast \phi \right) &= p_{123}^\ast p_{23}^\ast \phi \circ p_{123}^\ast  p_{12}^\ast \phi  = \pi_{23}^\ast \phi \circ \pi_{12}^\ast \phi, \quad \text{and}\\
p_{123}^\ast p_{13}^\ast \phi &= \pi_{13}^\ast \phi,
\end{align*}
where 
\[
\pi_{12}, \pi_{13}, \pi_{14}, \pi_{23}, \pi_{24}, \pi_{34} \colon S^{''''} \to S^{''}
\]
are the canonical morphisms. For $i,j,k \in \set{1,2,3,4}$ with $i < j < k$, define 
\[
\psi_{ijk} \coloneqq p_{ijk}^\ast \psi. 
\]
For instance, pulling back $\psi$ along $p_{123}$ gives a $2$-morphism
\begin{align*}
\psi_{123} = p_{123}^\ast\psi \colon \pi_{23}^\ast \circ \pi_{12}^\ast \phi \Rightarrow \pi_{13}^\ast \phi.
\end{align*}
Similarly, we obtain $2$-morphisms
\begin{align*}
\psi_{124} &\colon \pi_{24}^\ast \phi \circ \pi_{12}^\ast \phi \Rightarrow \pi_{14}^\ast \phi, \\
\psi_{134} &\colon \pi_{34}^\ast \phi \circ \pi_{13}^\ast \phi \Rightarrow \pi_{14}^\ast \phi, \\
\psi_{234} &\colon \pi_{34}^\ast \phi \circ \pi_{23}^\ast \phi \Rightarrow \pi_{24}^\ast \phi.
\end{align*}
Moreover, observe that under $p_{123}$, diagram \eqref{diagram:stacks-1} pulls back to the diagram

\begin{align*} 
\xymatrixcolsep{3pc}	
\xymatrixrowsep{3pc}
\xymatrix{
r_1^\ast \ca X' \ar[r]^{\pi_{12}^\ast \phi}\ar@{=}[d] & r_2^\ast \ca X' \ar@{=}[r] \ar@{}[dr] | {\refsymbol \;\;\;\;\;\;\;\;\;\;\;\;\;\;\;\;\;\;} & r_2^\ast \ca X' \ar[d]^{\pi_{23}^\ast\phi}&   \\
r_1^\ast \ca X' \ar[r]^{\pi_{13}^\ast \phi}&r_3^\ast \ca X'&r_3^\ast \ca X',\ar@{=}[l]& 
}
\end{align*}
in which the $2$-morphism $\Rightarrow$ is the $2$-morphism $\psi_{123}$ defined above (and with $r_i$ is as in \eqref{equation:ri}). Using pull-backs by the other three $p_{ijk} \colon S^{''''} \to S^{'''}$, we thus obtain four triangles, that we may put together to form the following tetrahedron:
\begin{align} 
\label{diagram:tetahedron}
\begin{split}
\xymatrix{
&&r_1^\ast \ca X'\ar[dd] \ar[dll] \ar[drr]&&\\
r_2^\ast \ca X'\ar[rrrr] |(.50)\hole \ar[drr]&&&&r_4^\ast \ca X'\\
&&r_3^\ast \ca X'.\ar[urr]&&
}
\end{split}
\end{align}

\begin{definition}[cf. Section 1.1 in \cite{romagny}]\label{def:2com}
A \emph{diagram} in a $2$-category $\mr C$ consists of a collection of objects, a collection of \( 1 \)-morphisms between certain pairs of objects, and a collection of \( 2 \)-morphisms between certain pairs of \( 1 \)-morphisms with the same source and target, such that for $i = 1,2$, the sets of \( i \)-morphisms in the diagram include all possible compositions of \( i \)-morphisms. A pair of \( i \)-morphisms in the diagram with the same source and target is called an \emph{\( i \)-circuit}; such a circuit is said to \emph{commute} if the two $i$-morphisms are equal. A diagram in \( \mr{C} \) is called \emph{\( i \)-commutative} if all of its \( i \)-circuits commute.
\end{definition}



\begin{definition} \label{definition:2-descent}
Let $p \colon S' \to S$ be a faithfully flat locally finitely presented morphism of schemes. Let $\mathcal X'$ be a stack in groupoids over $S'$. 
A \emph{$2$-descent datum} for $\ca X'$ over $S'$ consists of:
\begin{enumerate}
\item an isomorphism of stacks 
\[
\phi \colon p_1^\ast \ca X' \to p_2^\ast \ca X'
\]
over $S^{''}$;
\item a $2$-isomorphism 
$$
\psi \colon p_{23}^\ast \phi \circ p_{12}^\ast \phi \Rightarrow p_{13}^\ast\phi
$$
as in diagram \eqref{diagram:stacks-1};
\end{enumerate}
such that the diagram \eqref{diagram:tetahedron} is $2$-commutative (see Definition \ref{def:2com}). In other words, the $2$-morphisms $\psi_{ijk}$ between the several compositions \eqref{diagram:tetahedron} are compatible, in the sense that the following diagram of $2$-morphisms commutes:
\begin{align} \label{compatible-2morphisms}
\xymatrixcolsep{7pc}
\xymatrixrowsep{3pc}
\xymatrix{
\pi_{34}^\ast \phi \circ \pi_{23}^\ast \phi \circ \pi_{12}^\ast \phi  \ar@{=>}[r]^-{(\pi_{34}^\ast \phi)_\ast(\psi_{123}) }\ar@{=>}[d]^-{(\pi_{12}^\ast \phi)^\ast(\psi_{234}) }  & p_{34}^\ast \phi \circ p_{13}^\ast \phi \ar@{=>}[d]^-{\psi_{134}} \\
p_{24}^\ast \phi \circ p_{12}^\ast \phi \ar@{=>}[r]^-{\psi_{124}}   & p_{14}^\ast \phi.
}
\end{align}
\end{definition}


This gives the following result. 
\begin{theorem}[Breen] \label{theorem:breen}
Let $(\phi, \psi)$ be a $2$-descent datum for the stack $\ca X'$ over $S'$. Then there exists a stack $\ca X$ over $S$, an isomorphism 
\[
\rho \colon \ca X \times_SS' \xrightarrow{\sim} \ca X'
\]
of stacks over $S'$, and a $2$-isomorphism $\chi \colon p_2^\ast \rho \circ \textnormal{can} \Rightarrow \phi \circ p_1^\ast \rho$ as in diagram
\begin{align} \label{diagram:stacks-3}
\begin{split}
\xymatrix{
p_1^\ast(\ca X \times_S S') \ar[r]^{\textnormal{can}} \ar@{}[dr] | {\refsymbol} \ar[d]^{p_1^\ast \rho} & p_2^\ast(\ca X \times_S S') \ar[d]^{p_2^\ast \rho}\\
p_1^\ast \ca X' \ar[r]^\phi & p_2^\ast \ca X',
}
\end{split}
\end{align}
such that the natural compatibility condition between $\chi$ and $\psi$ is satisfied.  
\end{theorem}

\begin{proof}
This follows from \cite[Example 1.11.(i)]{breen}.
\end{proof}

To prove Theorem \ref{theorem:main:intro}, we use the fact that any stack over a scheme $S$ which is smooth locally on $S$ an algebraic stack, is actually an algebraic stack. More precisely, we have the following lemma, which is likely well-known but which we include for convenience of the reader. 
 

\begin{lemma} \label{lemma:local}Let $S' \to S$ be a morphism of schemes which faithfully flat and locally of finite presentation. Let $\ca X$ be a stack in groupoids over $S$ and define $\ca X' = \ca X \times_S S'$. Then the following holds. 
\begin{enumerate}
\item \label{lemma:stacklem-0} 
The diagonal  
$
\Delta \colon \ca X \to \ca X \times_S \ca X
$
is representable by algebraic spaces if and only if the diagonal $
\Delta' \colon \ca X' \to \ca X' \times_{S'} \ca X'
$ is representable by algebraic spaces. If this is true, then $\Delta$ is separated and quasi-compact if and only if $\Delta'$ is separated and quasi-compact.  
\item   \label{lemma:stacklem-1} 
The stack $\ca X'$ is an algebraic stack over $S'$ if and only if $\ca X$ is an algebraic stack over $S$. 
\item \label{lemma:stacklem-2} 
The stack $\ca X'$ is a Deligne--Mumford stack over $S'$ if and only if $\ca X$ is a Deligne--Mumford stack over $S$. 
\end{enumerate}
\end{lemma}
\begin{proof}
Let us first prove item \eqref{lemma:stacklem-0}. 
Since $\Delta'$ is the base change of $\Delta$ along $S' \to S$, we may assume that $\Delta'$ is representable, separated and quasi-compact, and it suffices under these conditions to prove that $\Delta$ is representable (indeed, being separated and quasi-compact is fppf local on the base, see \cite[\href{https://stacks.math.columbia.edu/tag/02YJ}{Tag 02YJ}]{stacks-project}). 
For this, it suffices to consider to schemes $U$ and $V$, equipped with morphisms $U \to \ca X$ and $V \to \ca X$, and prove that the fibre product $U \times_{\ca X} V$ is representable by an algebraic space, see \cite[Corollary 3.13]{LM-B}. Define $U' = \ca X' \times_{\ca X} U$ and $V'= \ca X' \times_{\ca X} V$. We obtain the following cartesian diagram:
{ \small
\[
\xymatrixrowsep{0.8pc}
\xymatrixcolsep{1pc}
\xymatrix{
& U' \times_{\ca X'} V'\ar[dd]  \ar[dl] \ar[drr]&& \\
V'\ar[dd] \ar[drr]&&& U \times_{\ca X} V \ar[dl] \ar[dd] \\
&U' \ar[dl] \ar[drr]& V\ar[dd] & \\
\ca X' \ar[drr] && &U. \ar[dl] \\
&& \ca X & 
}
\]
}
The morphism $\ca X' \to \ca X$ of stacks over $S$ is representable as it is the base change of the representable morphism $S' \to S$, hence $U'$ and $V'$ are representable by algebraic spaces. Since $\ca X'$ is an algebraic stack, 
the morphism $V' \to \ca X'$ is representable by algebraic spaces, which implies that its base change 
$
U' \times_{\ca X'} V' \to U'
$
is representable by algebraic spaces. Finally, the morphism of algebraic spaces $U' \to U$ is an fppf covering, hence an epimorphism. Using \cite[Lemme 4.3.3]{LM-B}, we conclude that  $U \times_{\ca X} V \to U$ is representable. As $U$ is scheme, $U \times_{\ca X} V$ is an algebraic space, proving item \eqref{lemma:stacklem-0}. 

To prove item \eqref{lemma:stacklem-1}, assume first that $\ca X$ is an algebraic stack over $S$. Then $\ca X'$ is an algebraic stack over $S$ by 
\cite[\href{https://stacks.math.columbia.edu/tag/04T2}{Tag 04T2}]{stacks-project}. In particular, $\ca X'$ is an algebraic stack over $S'$ by 
\cite[\href{https://stacks.math.columbia.edu/tag/04X5}{Tag 04X5}]{stacks-project}. Conversely, assume that $\ca X'$ is an algebraic stack over $S'$, and let $U'$ be an algebraic space equipped with a surjective and smooth morphism $U' \to \ca X'$. Then $\ca X'$ is an algebraic stack over $S$ by 
\cite[\href{https://stacks.math.columbia.edu/tag/04X5}{Tag 04X5}]{stacks-project}, and the morphisms $U' \to \ca X'$ and $\ca X' \to \ca X$ are representable by algebraic spaces. Therefore, the composition $U' \to \ca X' \to \ca X$ is representable by algebraic spaces, flat, surjective and locally of finite presentation. By Artin's theorem on representability of flat groupoids, see 
\cite[\href{https://stacks.math.columbia.edu/tag/06DC}{Tag 06DC}]{stacks-project}, it follows that $\ca X$ is an algebraic stack over $S$. 

Finally, we prove item \eqref{lemma:stacklem-2}. By item \eqref{lemma:stacklem-1} we may assume $\ca X$ (resp.\ $\ca X'$) is an algebraic stack over $S$ (resp.\ $S'$). 
As the morphism $\ca X' \to \ca X$ is representable by algebraic spaces, if $\ca X$ is a Deligne--Mumford stack over $S$ then $\ca X'$ is a Deligne--Mumford stack over $S'$.
Conversely, assume that $\ca X'$ is a Deligne--Mumford stack over $S'$. Notice that $\ca X'$ is an algebraic stack over $\Spec(\ZZ)$ by 
\cite[\href{https://stacks.math.columbia.edu/tag/04X5}{Tag 04X5}]{stacks-project}. Since $\ca X'$ is a Deligne--Mumford stack over $S'$, 
$\ca X'$ is a Deligne--Mumford stack, that is, a Deligne--Mumford stack over $\Spec(\ZZ)$. Hence, by 
\cite[\href{https://stacks.math.columbia.edu/tag/06MB}{Tag 06MB}]{stacks-project}, $\ca X'$ is DM in the sense of 
\cite[\href{https://stacks.math.columbia.edu/tag/050D}{Tag 050D}]{stacks-project}, which implies that $\ca X'$ is DM over $S'$ by 
\cite[\href{https://stacks.math.columbia.edu/tag/050N}{Tag 050N}]{stacks-project}. Thus, by
\cite[\href{https://stacks.math.columbia.edu/tag/06TZ}{Tag 06TZ}]{stacks-project}, the algebraic stack $\ca X$ is DM over $S$. Consequently, $\ca X$ is DM by 
\cite[\href{https://stacks.math.columbia.edu/tag/050N}{Tag 050N}]{stacks-project}, hence Deligne--Mumford by 
\cite[\href{https://stacks.math.columbia.edu/tag/06N3}{Tag 06N3}]{stacks-project}, hence Deligne--Mumford over $S$.
\end{proof}



\begin{proof}[Proof of Theorem \ref{theorem:main:intro}]
Theorem \ref{theorem:breen} yields the stack $\ca X$ over $S$ together with the $1$-isomorphism $\rho \colon \ca X \times_S S' \xrightarrow{\sim} \ca X'$ and the $2$-isomorphism $\chi \colon p_2^\ast \rho \circ \textnormal{can} \Rightarrow \phi \circ p_1^\ast \rho$ that have the right compatibility properties with respect to $\psi$. 
The remaining assertions follow from Lemma \ref{lemma:local}. 
\end{proof}


\begin{corollary} \label{cor:corollary}
Let $S' \to S$ be a morphism of schemes, faithfully flat and locally of finite presentation, and let $X'$ be a scheme over $S'$ equipped with a descent datum $\phi$ as in Section \ref{section:descendingschemes}. Then there exists an algebraic space $X$ over $S$ and an $S$-morphism of algebraic spaces $\pi \colon X' \to X$ such that the induced map $X' \to X \times_{S}S'$ is an isomorphism of algebraic spaces with descent data. 
\end{corollary}
\begin{proof}
Theorem \ref{theorem:main:intro} implies the existence of $X$ as a Deligne--Mumford stack, hence we only need to prove that $X$ is an algebraic space. For this, in view of  \cite[Proposition 2.4.1.1]{LM-B}, it suffices to show that the diagonal $\Delta_{X/S} \colon X \to X \times_S X$ is a monomorphism. This is fppf local on $S$ \cite[\href{https://stacks.math.columbia.edu/tag/02YK}{Tag 02YK}]{stacks-project}, thus follows from the fact that $X' \to X' \times_{S'}X'$ is a monomorphism. 
\end{proof}

\subsection{Galois $2$-descent}

Recall the notion of Galois $2$-descent datum, see Definition \ref{def:Galois2descent}. 

\begin{lemma} \label{lemma:set-can-bij}
Let $S' \to S$ be a surjective finite étale morphism of schemes which is a Galois covering with Galois group $\Gamma$, and let $\ca X'$ be a stack over $S'$. Then the following sets are in canonical bijection:
\begin{enumerate}
\item \label{set1}The set of $2$-descent data 
for $\ca X'$ over $S'$. 
\item \label{set2}The set of Galois $2$-descent data 
for $\ca X'$ over $S'$. 
\end{enumerate}
\end{lemma}
\begin{proof}
See \cite[Section 6.2, Example B]{BLR} and \cite[Proposition 4.4.4]{poonen} for the proof in the case of schemes. 
In the stacky case, one proceeds as follows.

As $S' \to S$ is a Galois covering with Galois group $\Gamma$, the maps 
\begin{align}
\label{galoisiso-first}
\Gamma \times S' \to S'', \quad \quad &(\sigma, x) \mapsto (\sigma x, x), \\
\label{galoisiso-second}
\Gamma \times \Gamma \times S' \to S''', \quad \quad &(\tau, \sigma, x) \mapsto ( (\tau \sigma)x, \sigma x, x), \quad \quad \text{and} \\
\label{galoisiso-third}
\Gamma \times \Gamma \times \Gamma \times S' \to S'''', \quad \quad &(\gamma, \tau, \sigma, x) \mapsto ((\gamma \tau \sigma)x, (\tau \sigma)x, \sigma x, x),
\end{align}
are isomorphisms. Furthermore, we have canonical isomorphisms 
\begin{align}
\label{line:first}
p_1^\ast \ca X' & = \coprod_{\sigma \in \Gamma}{^\sigma \ca X'}, \quad 
p_2^\ast \ca X' = \coprod_{\sigma \in \Gamma} {\ca X'}, \\
\label{line:second}
q_1^\ast \ca X' &  
= \coprod_{(\tau, \sigma) \in \Gamma \times \Gamma} {^{\tau\sigma} \ca X'}, \quad q_2^\ast = \coprod_{(\tau, \sigma) \in \Gamma \times \Gamma} {^\sigma \ca X'}, \quad q_3^\ast \ca X' = \coprod_{(\tau, \sigma) \in \Gamma \times \Gamma} \ca X'.
\end{align}
The isomorphisms in \eqref{line:first} are isomorphisms of $S''$-stacks, where $\coprod_{\sigma \in \Gamma} {\ca X'}$ and $\coprod_{\sigma \in \Gamma} {^\sigma \ca X'}$ as stacks over $S''$ via their canonical morphisms to $\coprod_{\sigma \in \Gamma} S'$ and the isomorphism $S'' = \coprod_{\sigma \in \Gamma} S'$ of \eqref{galoisiso-first}. Similarly, we view $\coprod_{(\tau, \sigma) \in \Gamma \times \Gamma} \ca X'$, $\coprod_{(\tau, \sigma) \in \Gamma \times \Gamma} {^\sigma \ca X'}$ and $\coprod_{(\tau, \sigma) \in \Gamma \times \Gamma} {^{\tau \sigma} \ca X'}$ as stacks over $S'''$ via their canonical morphisms to $\coprod_{(\tau, \sigma) \Gamma \times \Gamma} S'$ and the identification $S''' = \coprod_{(\tau, \sigma) \in \Gamma \times \Gamma} S'$ of \eqref{galoisiso-second}; then the isomorphisms in \eqref{line:second} are isomorphisms stacks over $S'''$. We conclude:
\begin{enumerate} 
\item[(i)] \label{item:firstgiving} Giving an isomorphism 
$
\phi \colon p_1^\ast \ca X' \to p_2^\ast \ca X'
$
of stacks over $S''$ is equivalent to giving for each $\sigma \in \Gamma$ an isomorphism 
$
f_\sigma \colon {^\sigma \ca X'} \to \ca X'
$
of stacks over $S'$. 
\item[(ii)] 
\label{item:secondgiving}
Let $
\phi \colon p_1^\ast \ca X' \to p_2^\ast \ca X'
$ be an isomorphism corresponding to isomorphisms $
f_\sigma \colon {^\sigma \ca X'} \to \ca X'
$ $(\sigma \in \Gamma)$ as in (i). Then the isomorphisms $$p_{12}^\ast \phi \colon q_1^\ast \ca X' \to q_2^\ast \ca X', \quad p_{23}^\ast \phi \colon q_2^\ast \ca X' \to q_3^\ast \ca X', \quad p_{13}^\ast \phi \colon q_1^\ast \ca X' \to q_3^\ast \ca X'$$ induced by $\phi$
 correspond to the isomorphisms $${^\sigma f_\tau} \colon {^{\tau \sigma} \ca X'} \to {^\sigma \ca X'}, \quad f_\sigma \colon {^\sigma \ca X'} \to \ca X', \quad f_{\tau \sigma} \colon {^{\tau \sigma} \ca X'} \to \ca X' \quad ((\tau, \sigma) \in \Gamma \times \Gamma)$$ induced by the $f_\sigma$ $(\sigma \in \Gamma)$. In particular, giving a $2$-isomorphism $\psi \colon p_{23}^\ast \phi \circ p_{12}^\ast \phi \Rightarrow p_{13}^\ast\phi$ as in diagram \eqref{diagram:stacks-1} is equivalent to giving for each $(\tau, \sigma) \in \Gamma \times \Gamma$ a $2$-isomorphism $f_\sigma \circ {^\sigma f_\tau} \Rightarrow f_{\tau \sigma}$. 
\item[(iii)] Let $
\phi \colon p_1^\ast \ca X' \to p_2^\ast \ca X'
$ be an isomorphism corresponding to isomorphisms $
f_\sigma \colon {^\sigma \ca X'} \to \ca X'
$ $(\sigma \in \Gamma)$ as in (i), and let $\psi \colon p_{23}^\ast \phi \circ p_{12}^\ast \phi \Rightarrow p_{13}^\ast\phi$ be a $2$-isomorphism corresponding to $2$-isomorphisms $f_\sigma \circ {^\sigma f_\tau} \Rightarrow f_{\tau \sigma}$ $\left((\tau, \sigma) \in \Gamma \times \Gamma \right)$ as in (ii). Then the diagram \eqref{compatible-2morphisms} commutes if and only if the diagram \eqref{galoisdiagram2morphisms} commutes. 
\end{enumerate}
This provides a canonical bijection between the sets \eqref{set1} and \eqref{set2}. 
\end{proof}

\begin{proof}[Proof of Theorem \ref{theorem:galois}]
See Theorem \ref{theorem:main:intro} and Lemma \ref{lemma:set-can-bij}. 
\end{proof}

\subsection{Group actions and Galois $2$-descent}
Let $S$ be a scheme and $G$ a group scheme over $S$, with multiplication map $m \colon G \times_S G \to G$ and unit section $e \colon S \to G$. Let $\ca X$ be a stack over $S$. 
We define a \emph{group action of $G$ on $\ca X$ over $S$} 
as an action of the functor in groups over $S$ associated to $G$ on the stack $\ca X$ over $S$, 
in the sense of \cite[Definition 1.3]{romagny} (see also \cite[Definition 6.1]{breen-bitorseurs}, \cite[Section 2.1]{laszlo}, and \cite[Definition 5.1]{noohi-groupactions} for slightly differently formulated but equivalent notions). For convenience of the reader, we recall that such a group action 
consists of a triple $(\mu, \alpha, \beta)$, where $\mu$ is a morphism of stacks $\mu \colon G \times_S \ca X \to \ca X$ over $S$, and $\alpha$ and $\beta$ are $2$-morphisms $\alpha \colon \mu \circ ( \id_{G} \times \mu) \Rightarrow \mu \circ (m \times \id_{\ca X})$ and $\beta \colon \mu \circ (e, \id_{\ca X}) \Rightarrow \id_{\ca X}$, as in the following diagrams:
\begin{align}
\label{GGX}
\begin{split}
\xymatrixrowsep{3pc}
\xymatrixcolsep{3pc}
\xymatrix{
G \times_S G \times_S \ca X \ar[r]^-{m \times \id_{\ca X}} \ar[d]_-{\id_{G} \times \mu} & G \times_S \ca X \ar[d]^-\mu \\
G \times_S \ca X \ar[r]_-\mu
\ar@{}[ur] | {\rotatebox{180}{\refsymbol} \;\;\; \alpha}
 & \ca X,
} 
\xymatrixrowsep{3pc}
\xymatrixcolsep{3pc}
\xymatrix{
&G \times_S \ca X 
\ar@{}[d] | {\rotatebox{45}{\refsymbol}\beta}
 \ar[dr]^-{\mu} &  \\
\ca X\ar[ur]^-{(e, \id_{\ca X})} \ar[rr]_-{\id_{\ca X}} && \ca X,
}
\end{split}
\end{align}
that satisfy the following compatibility condition. For each $S$-scheme $T$, each $x,y \in \ca X(T)$ and each $g,h \in G(T)$, $\alpha$ and $\beta$ are given by natural isomorphisms  
\[
\alpha^x_{g,h} \colon \mu(g, \mu(h,x)) \xrightarrow{\sim} \mu(gh,x) \quad \text{ and } \quad \beta^x \colon \mu(e,x) \xrightarrow{\sim} x,
\]
and one requires that for all $T \in (\Sch/S), x \in \ca X(T)$ and $g,h,k \in G(T)$, we have
\begin{align}
\label{align:first}
\begin{split}
\alpha_{g,hk}^x \circ \mu\left(g, \alpha_{h,k}^x\right) = \alpha_{gh,k}^x \circ \alpha^{\mu(k,x)}_{g,h}, \quad 
\mu(g, \beta^x) = \alpha_{g,e}^x, \quad \beta^{\mu(h,x)} = \alpha_{e,h}^x.
\end{split} 
\end{align}
\begin{example} \label{example:stackyness}
Let $\ca X$ be a stack over a scheme $S$, and let $\Gamma$ be a finite group. Then an action of $\Gamma$ on $\ca X$ over $S$ is given by an isomorphism $\mu(\sigma) \colon \ca X \to \ca X$ for each $\sigma \in \Gamma$, a $2$-isomorphism $\alpha_{\tau, \sigma} \colon \mu(\tau) \circ \mu(\sigma) \Rightarrow \mu(\tau\sigma)$ for each $\tau, \sigma \in \Gamma$, and a $2$-isomorphism $\beta \colon \mu(e) \Rightarrow \id_{\ca X}$, such that the composition $\mu(\gamma) \circ \mu(\tau) \circ \mu(\sigma) \Rightarrow \mu(\gamma \tau) \circ \mu(\sigma) \Rightarrow \mu(\gamma \tau \sigma)$ coincides with the composition $\mu(\gamma) \circ \mu(\tau) \circ \mu(\sigma) \Rightarrow \mu(\gamma) \circ \mu(\tau \sigma) \Rightarrow \mu(\gamma \tau \sigma)$ for $\gamma, \tau, \sigma \in \Gamma$, such that $\mu(\sigma)_\ast(\beta) = \alpha_{\sigma,e} \colon \mu(\sigma) \circ \mu(e) \Rightarrow \mu(\sigma)$ for $\sigma \in \Gamma$, and such that $\mu(\sigma)^\ast(\beta) = \alpha_{e, \sigma} \colon \mu(e) \circ \mu(\sigma) \Rightarrow \mu(\sigma)$ for $\sigma \in \Gamma$.  
\end{example}

\begin{lemma} \label{lemma:action-and-descent}
Let $S' \to S$ be a surjective finite étale morphism which is a Galois cover with Galois group $\Gamma$, and let $\ca X'$ be a stack over $S'$. 
Then a group action of $\Gamma$ on the stack $\ca X'$ over $S$, such that for each $\sigma \in \Gamma$ the diagram
\begin{align} \label{diagram:XsigmaScommutes}
\begin{split}
\xymatrix{
\ca X' \ar[d]  \ar[r]^-\sigma & \ca X' \ar[d] \\
S' \ar[r]^-\sigma & S'
}
\end{split}
\end{align}
commutes, determines in a canonical way a Galois $2$-descent datum for $\ca X'$. 
\end{lemma}
\begin{proof}
Let $\Gamma$ act on $\ca X'$ such that \eqref{diagram:XsigmaScommutes} commutes for each $\sigma \in \Gamma$. For $\sigma \in \Gamma$, denote by $\mu(\sigma) \colon \ca X' \to \ca X'$ the attached isomorphism. By Example \ref{example:stackyness}, the $2$-isomorphisms $\alpha$ and $\gamma$ in the diagrams \eqref{GGX} are given by $2$-isomorphisms
\begin{align*}
\alpha_{\tau, \sigma} \colon \mu(\tau) \circ \mu(\sigma) \Rightarrow \mu(\tau \circ \sigma) \quad \left((\tau, \sigma) \in \Gamma \times \Gamma \right), \quad \quad \text{and} \quad \quad \beta \colon \mu(e) \Rightarrow \id_{\ca X}. 
\end{align*}
Denote by $\can_\sigma \colon {^\sigma \ca X'} \to \ca X'$ the base change of $\sigma \colon S' \to S'$ along $\ca X' \to S'$, and define $f_\sigma \coloneqq \mu(\sigma^{-1}) \circ \can_\sigma$. This gives an isomorphism $f_\sigma \colon {^\sigma \ca X'} \to \ca X'$ of stacks over $S'$. For $(\tau, \sigma) \in \Gamma \times \Gamma$, consider the following diagram, in which $\can_{\tau, \sigma} \colon {^{\tau \sigma} \ca X'} \to {^\tau \ca X'}$ is the base change of $\sigma \colon S' \to S'$ along ${^\tau \ca X'} \to S'$:
\begin{align}\label{diagram:GammaGamma}
\begin{split}
\xymatrix{
\ca X'  & & \\
{^\sigma \ca X'} \ar[u]^-{f_\sigma} \ar[r]^-{\can_\sigma} & \ca X' \ar[ul]_-{\mu(\sigma^{-1})}  \ar@{}[ur] | {\rotatebox{180}{\refsymbol} \; \alpha_{\sigma^{-1}, \tau^{-1}}} & \\
{^{\tau \sigma} \ca X'}\ar[d] \ar[u]^-{^\sigma f_\tau}  \ar[r]^-{\can_{\tau, \sigma}} & {^\tau \ca X'} \ar[u]^-{f_\tau} \ar[d] \ar[r]^{\can_\tau} & \ca X' \ar[d] \ar[ul]_-{\mu(\tau^{-1})} \ar@/_5pc/[uull]_-{\mu(\sigma^{-1}\tau^{-1})} \\
S' \ar[r]^-\sigma & S' \ar[r]^-\tau & S'. 
}
\end{split}
\end{align}
 We obtain canonical $2$-isomorphisms
 \begin{align*}
 f_\sigma \circ {^\sigma f_\tau} &= \mu(\sigma^{-1}) \circ \can_\sigma \circ {^\sigma f_\tau} \Rightarrow \mu(\sigma^{-1}) \circ f_\tau \circ \can_{\tau, \sigma} \\
 &\Rightarrow \mu(\sigma^{-1}) \circ \mu(\tau^{-1}) \circ \can_\tau \circ \can_{\tau, \sigma} \Rightarrow \mu(\sigma^{-1}) \circ \mu(\tau^{-1}) \circ \can_{\tau \sigma} \\ &\Rightarrow \mu(\sigma^{-1} \tau^{-1}) \circ \can_{\tau\sigma} = f_{\tau \sigma},
\end{align*}
whose composition we denote by 
\[
\psi_{\tau,\sigma} \colon f_\sigma \circ {^\sigma f_\tau} \Rightarrow f_{\tau \sigma}. 
\]
 Then for $(\gamma, \tau, \sigma) \in \Gamma \times \Gamma \times \Gamma$, we consider the following cube of $2$-morphisms:
{\footnotesize
\begin{align} \label{diagramoftwomorphisms}
\begin{split}
\xymatrixrowsep{1.5pc}
\xymatrixcolsep{0.07pc}
\xymatrix{
& \mu(\gamma^{-1}) \circ \mu(\tau^{-1}) \circ \mu(\sigma^{-1}) \circ \can_{\gamma \tau \sigma} \ar@{=>}[dd]  \ar@{=>}[dl] \ar@{=>}[drr]&& \\
f_\sigma \circ {^\sigma f_\tau} \circ {^{\tau \sigma} f_\gamma}\ar@{=>}[dd] \ar@{=>}[drr]&&& \mu(\gamma^{-1}) \circ \mu(\tau^{-1}\sigma^{-1}) \circ \can_{\gamma \tau \sigma} \ar@{=>}[dl] \ar@{=>}[dd] \\
& \mu(\gamma^{-1}\tau^{-1}) \circ \mu(\sigma^{-1}) \circ \can_{\gamma \tau \sigma} \ar@{=>}[dl] \ar@{=>}[drr]& f_{\tau \sigma} \circ {^{\tau \sigma} f_\gamma}\ar@{=>}[dd] & \\
f_\sigma \circ {^\sigma f_{\tau \gamma}} \ar@{=>}[drr] && &\mu(\gamma^{-1}\tau^{-1}\sigma^{-1}) \circ \can_{\gamma \tau \sigma}. \ar@{=>}[dl] \\
&& f_{\sigma \tau \gamma} & 
}
\end{split}
\end{align}
}
The face on the back of the cube \eqref{diagramoftwomorphisms} commutes because of \eqref{align:first}. In fact, each face of \eqref{diagramoftwomorphisms}, 
except possibly the face on the front, commutes. Thus, the face on the front of  \eqref{diagramoftwomorphisms} commutes as well. That is, diagram \eqref{galoisdiagram2morphisms} commutes, hence $(f_\sigma \; (\sigma \in \Gamma), \psi_{\tau, \sigma} \; ((\tau, \sigma) \in \Gamma \times \Gamma))$ is a Galois $2$-descent datum.
\end{proof}

\section{Morphisms of stacks with descent data}

Let $S' \to S$ be a faithfully flat locally finitely presented morphism of schemes. For $i = 1,2$, let $\ca X_i'$ be a stack over $S'$, equipped with a $2$-descent datum $(\phi_i, \psi_i)$, see Definition \ref{definition:2-descent}.

\begin{definition} \label{def:morphism}
A \emph{morphism}
$$
\left( \ca X_1', \phi_1, \psi_1 \right) \to \left( \ca X_2', \phi_2, \psi_2 \right) 
$$
of stacks with $2$-descent data over $S'$ consists of a pair $(f, \alpha)$, where $f \colon \ca X_1' \to \ca X_2'$ is a morphism of stacks over $S'$ and $\alpha \colon  p_2^\ast f \circ \phi_1 \Rightarrow \phi_2 \circ p_1^\ast f$ is a $2$-morphism as in the following diagram:
\begin{align*} 
\begin{split}
\xymatrixcolsep{2.8pc}	
\xymatrixrowsep{2.8pc}
\xymatrix{
p_1^\ast \ca X'_1 \ar[r]^-{\phi_1} \ar@{}[dr] | {\refsymbol} \ar[d]^{p_1^\ast f} & p_2^\ast \ca X'_1 \ar[d]^{p_2^\ast f}\\
p_1^\ast \ca X'_2 \ar[r]^-{\phi_2} & p_2^\ast \ca X'_2,
}
\end{split}
\end{align*}
such that the following diagram in the category of stacks over $S^{'''}$ is $2$-commutative (in the sense of Definition \ref{def:2com}):
{\small
\begin{align*} 
\begin{split}
\xymatrixcolsep{2.4pc}	
\xymatrixrowsep{3.4pc}
\xymatrix@1{
&p_{12}^\ast p_1^\ast \ca X'_1 \ar[dl]_{p_{12}^\ast p_1^\ast f} \ar[rr]^{p_{12}^\ast \phi_1}\ar@{=}[dd]  && p_{12}^\ast p_2^\ast \ca X'_1 \ar[dl]_-{p_{12}^\ast p_2^\ast f} \ar@{=}[rr] \ar@{}[dr]  && p_{23}^\ast p_1^\ast \ca X'_1 \ar[dl]_-{p_{23}^\ast p_1^\ast f} \ar[dd]^{p_{23}^\ast\phi_1}  \\
p_{12}^\ast p_1^\ast \ca X'_2 \ar[rr]^{\;\;\;\;\;\;\;\;\;\;\;\;\;\;\;\;\;\;\;\; p_{12}^\ast \phi_2}\ar@{=}[dd] && p_{12}^\ast p_2^\ast \ca X'_2 \ar@{=}[rr] \ar@{}[dr] && p_{23}^\ast p_1^\ast \ca X'_2 \ar[dd]^(.3){p_{23}^\ast\phi_2}&  \\
&p_{13}^\ast p_1^\ast \ca X'_1 \ar[dl]^(.3){p_{13}^\ast p_1^\ast f} \ar[rr]^{p_{13}^\ast \phi_1}&&p_{13}^\ast p_2^\ast \ca X'_1 \ar[dl]^(.4){p_{13}^\ast p_2^\ast f}&&p_{23}^\ast p_2^\ast \ca X'_1 \ar[dl]^(.4){p_{23}^\ast p_2^\ast f} \ar@{=}[ll]  \\
p_{13}^\ast p_1^\ast \ca X'_2 \ar[rr]^{p_{13}^\ast \phi_2}&&p_{13}^\ast p_2^\ast \ca X'_2&&p_{23}^\ast p_2^\ast \ca X'_2. \ar@{=}[ll]&
}
\end{split}
\end{align*}
}
Here, the $2$-morphisms in each square are the canonical ones (induced by $\alpha$).
\end{definition}

\begin{proof}[Proof of Proposition \ref{proposition:morphisms}]
This follows from \cite[Chapitre II, \S 2.1.5]{giraud} (and is also a special case of Theorem \ref{thm:breen}). 
\end{proof}


\section{Example}
Let $k$ be a field and let $k \subset k'$ be a degree two separable field extension. 
Let $\sigma \in \Gal(k'/k)$ be the generator of the Galois group $\Gal(k'/k)$ of $k'$ over $k$. Let $\ca X'$ be a stack over $k'$ equipped with a $1$-isomorphism $$\sigma \colon \ca X' \to \ca X'$$ of stacks over $k$, and a $2$-isomorphism $\alpha \colon \sigma^2 \Rightarrow \id_{\ca X'}$ between $\sigma^2$ and the identity functor. 
One obtains the descended stack $\ca X$ over $k$ by defining, for $T \in (\Sch/k)$, $\ca X(T)$ as the groupoid of pairs $(x, \varphi)$ with $x \in \ca X'(T_{k'})$ and $\varphi \colon x \to \sigma(x)$ an isomorphism such that the composition $$x \xrightarrow{\varphi} \sigma(x) \xrightarrow{\sigma(\varphi)} \sigma^2(x) \xrightarrow{\alpha^x} x$$ is the identity. There is a natural isomorphism $\ca X' \cong \ca X \times_k k'$ of stacks over $k'$. 


 
\printbibliography
\end{document}